\documentclass[12pt]{amsart}
\usepackage{amsmath,amsfonts,amssymb,amsthm}
\usepackage{amsmath,amstext,amsbsy,amssymb}
\usepackage{epsfig}
\usepackage{geometry} 
\newtheorem{theorem}{Theorem}[section]
\newtheorem{definition}{Definition}[section]
\newtheorem{example}{Example}[section]

\newtheorem{corollary}{Corollary}[section]

\newtheorem{proposition}{Proposition}[section]

\newtheorem{lemma}{Lemma}[section]
\newtheorem{remark}{Remark}[section]

\newcommand{\pfq}       {{P\hspace{-.8mm}f}_q}  
\newcommand{\pfc}       {{P\hspace{-.8mm}f}}        
\newcommand{\GL}        {GL(n,\mathbb{C})}      
\newcommand{\GLq}       {GL_q(n,\mathbb{C})}    
\newcommand{\GLqt}       {GL_q(2n,\mathbb{C})}    
\newcommand{\SP}        {Sp(2n,\mathbb{C})}      
\newcommand{\SPq}       {Sp_q(2n,\mathbb{C})}   

\newcommand{\Uspq}  {U_q(\mathfrak{sp}(2n,\mathbb{C}))}
\newcommand{\Usp}       {U(\mathfrak{sp}(2n,\mathbb{C}))}
\newcommand{\Uglq}  {U_q(\mathfrak{gl}(n,\mathbb{C}))}
\newcommand{\Uq}        {U_q(\mathfrak{g})}
\newcommand{\AXq}       {A_q(X)}
\newcommand{\AGq}   {A(G)}
\newcommand{\Aq}        {\mathcal{A}_q}
\newcommand{\ARq}   {A_q^R(\mathcal{A})}
\newcommand{\ALq}       {A_q^L(\mathcal{A})}
\newcommand{\ABplus}    {A(B_{+})}
\newcommand{\ABminus}   {A(B_{-})}

\begin{document}

\title{Zonal Polynomials and Quantum Antisymmetric Matrices}
\author{Naihuan Jing and Robert Ray}

\thanks{{\scriptsize
\hskip -0.4 true cm MSC (2010): Primary: 20G42;  Secondary: 17B37, 43A90, 05E10.
\newline Keywords: Pfaffians, quantum groups, invariants, quantum anti-symmetric matrices.\\
}}

\maketitle

\begin{abstract}
We study the quantum symmetric spaces for quantum general linear groups modulo
symplectic groups. We first determine the structure of the quotient quantum group and
completely determine the quantum invariants. We then derive the characteristic property
for quantum Phaffian as well as its role in the quantum invariant sub-ring. The spherical functions, viewed as
Macdonald polynomials, are also studied as the quantum analog of zonal spherical polynomials.
\end{abstract}

\section{Introduction}
The regular representation of $\GL$ can be realized on the ring
\begin{equation}
    A(X) = \mathbb{C}[x_{11},x_{12},\ldots, x_{nn}]
\end{equation}
where regular functions are polynomials of the matrix elements of
the $n \times n$ matrices. It is well known that $A(X)$ is a completely reducible $\GL$-module
and the associated irreducible
polynomial sub-representations are parametrized by the set of partitions
\begin{equation}
    P_n = \left\{ \lambda = (\lambda_1, \ldots, \lambda_n) \in \mathbb{Z}^n \;;\; \lambda_1 \ge \cdots \ge \lambda_n \ge 0    \right\}.
\end{equation}
For a given $\lambda \in P_n$, there is a unique (up to isomorphism)
irreducible representation $V(\lambda)$ with highest weight
$\lambda$.  Similarly, one considers the modules
$A(Sym(n))$ of the polynomials in the coordinates of the $n \times n$
symmetric matrix, and the module $A(Skew(2n))$ of the polynomials in the
coordinates of the $2n \times 2n$ skew symmetric matrix.  These
representations decompose into the multiplicity free sums \cite{Hu, Ho}:
\begin{align}
    A(Sym(n)) & \simeq \bigoplus_{\lambda \in P_n} V(2\lambda)\\
    A(Skew(2n)) & \simeq \bigoplus_{\lambda \in P_{n}} V(\lambda_1,\lambda_1,\lambda_2, \lambda_2,
    \ldots, \lambda_{n},\lambda_{n}),
\end{align}
which are invariant under the action of $O(n,\mathbb{C})$ and $\SP$ respectively.
As L. Hua first noticed and A. James later formulated that the $O_n$ and $Sp_n$ invariants are
one dimensional and the
zonal spherical functions enjoy similar properties of Schur symmetric functions \cite{Hu, J, Mac}.

In the case of quantum analog of the symmetric pair of general linear groups and symplectic groups,
Noumi and Letzter
\cite{N, L} showed that the quantum spherical functions are indeed certain Macdoanald symmetric functions
by working on the quantum algebra of the enveloping algebras.
We will study directly the quantum invariant ring as a subring of the quantum general linear group.
As in \cite{JY}, we compute the Hopf ideal of quantum invariants for the symplectic case using
certain quadratic polynomials of matrix coefficients of quantum general linear groups.

A new feature in current work on quantum invariants is that we
will study the important role played
by Pfaffian as in the classical symplectic case. In the quantum case, the quantum Phaffian
played an important role in the invariant theory as well \cite{S}. We first give a closed form definition for
the quantum Phaffian
and study its representation-theoretic meaning in the quantum setting.
Through this we are able to give an appropriate quantum analog of
its relations with quantum determinant. As expected,
quantum Phaffians enjoy similar properties as quantum determinant in the orthogonal case.

This paper is organized as follows.
In Section \ref{QuantumGroups} we first recall some basic facts of certain quantum algebras, in particular, we discuss a quantum deformation of $A(X)$ and the associated quantum version of $\GL$ as presented in Noumi, Yamada, and Mimachi \cite{NYM93} and we recall the quantized universal enveloping algebra $\Uglq$.  In Section \ref{SympecticInvariants} we describe a quantum symplectic group, $\SPq$.  Since there does not seem to be a natural embedding of  $\SPq$ in $\GLqt$ we define $\SPq$ invariants (left and right) in an infinitesimal manner, similar to an earlier construction by Jing and Yamada \cite{JY} of polynomial invariants for a quantum orthogonal group.  These quantum symplectic invariants give us a quantum version of the regular functions of the
antisymmetric matrices. In addition to defining the generators of these functions, we describe their relations and we discuss a construction of a quantum analog to the Pfaffian function.

We then describe a complete reduction of the $\SPq$ invariant spaces (left and right) into irreducible
modules and we follow with a construction and characterization of the associated bi-invariant space and its basis of zonal polynomials. In the last section, a connection between the zonal polynomials and certain Macdonald polynomials is discussed.

\section{Quantum Groups}\label{QuantumGroups}

Quantum groups are defined as certain one-parameter deformation of
the algebra of algebraic functions on simple Lie groups \cite{FRT}.
In other words, we will describe $\AXq$ to be like the classical
algebra $A(X)$, except with noncommuting relations imposed upon its
generators.  Throughout the paper we will let $q$ to be a complex
number and for $q \ne 1$ we require that $q$ not a root of unity.

\subsection{$\AXq$, $\AGq$ and $\GLq$}
We first define the algebra of functions $\AXq$ on $X =
Mat_q(n,\mathbb{C})$ as a noncommutative
$\mathbb{C}$-algebra
\begin{equation}
\AXq = \mathbb{C}_q\left[x_{11},x_{12}, \ldots , x_{nn} \right].
\end{equation}
generated by $x_{11},x_{12}, \ldots , x_{n,n}$ and with relations
\begin{align*}
    x_{ik}x_{jk} =&\;qx_{jk}x_{ik}, \quad
    x_{ki}x_{kj} =\;qx_{kj}x_{ki},\\
    x_{il}x_{jk} =&\;x_{jk}x_{il},\\
    x_{ik}x_{jl}-&x_{jl}x_{ik}=\;\left(q-q^{-1}\right)x_{il}x_{jk},
\end{align*}
where $i<j$ and $k<l$. The relations can be visualized by the
diagram (see Figure \ref{X_Relations}) with a ``square" of
generators.
\begin{figure}[h]\label{X_Relations}
\begin{center}
\setlength{\unitlength}{1cm}

\begin{picture}(3,3)
    \put(0,0){\makebox(1,1){$x_{jk}$}}
    \put(0,2){\makebox(1,1){$x_{ik}$}}
    \put(2,0){\makebox(1,1){$x_{jl}$}}
    \put(2,2){\makebox(1,1){$x_{il}$}}

    \put(1,0.5){\vector(1,0){1}}
    \put(1,2.5){\vector(1,0){1}}
    \put(0.5,2){\vector(0,-1){1}}
    \put(2.5,2){\vector(0,-1){1}}
\end{picture}
\caption{$\AXq$ Relations, $x \rightarrow y$ implies $xy = qyx$}
\end{center}
\end{figure}

$\AXq$ is a bialgebra using the same coproduct and counit maps as
defined on $A(X)$,  see \cite{Mo}.

Let $I$ and $J$ be two subsets of $\{ 1,2, \ldots, n \}$ with $\#I = \#J = r$ with ordered elements,
i.e. $i_1 < i_2 < \ldots <i_r \in I$ and $j_1 < j_2 < \ldots j_r \in J$.  The quantum $r$-minor
determinants are defined as
\begin{equation}
    \xi_J^I = \xi_{j_1,\ldots,j_r}^{i_1,\ldots,i_r} = \sum_{\sigma \in
    \mathfrak S_r}(-q)^{l(\sigma)}x_{i_1j_{\sigma(1)}}x_{i_2j_{\sigma(2)}} \ldots x_{i_r j_{\sigma(r)}}
\end{equation}
where $l(\sigma)$ denotes the number of pairs $(i,j)$ with $i<j$ and
$\sigma(i) > \sigma(j)$. There is a unique quantum $n$-minor
determinant, and it is denoted by $det_q$ \cite{JY}. We define the
algebra of regular functions $\AGq$ on the quantum group $\GLq$ by
adjoining $det_q^{-1}$ to $\AXq$
\begin{equation}
    \AGq = \left[x_{11},x_{12}, x_{13},\ldots, x_{nn}, det_q^{-1} \right]
\end{equation}
Then, $\GLq$ is defined as the spectrum of the Hopf algebra algebra $\AGq$, i.e.
\begin{equation}
    \GLq=Spec(\AGq)
\end{equation}
one usually refers simply $\GL$ as $\AGq$.

In addition to the relations of $\AXq$, $\AGq$ also has the following relations \cite{NYM93}
\begin{equation}
    x_{ij}\cdot det_q^{-1} = det_q^{-1} \cdot x_{ij}
\end{equation}
\begin{equation}
    det_q^{-1} \cdot det_q = det_q \cdot det_q^{-1} = 1.
\end{equation}
This allows us to define the algebra morphism $S:\AGq \rightarrow \AGq$ by
\begin{equation}
    S(x_{ij}) = (-q)^{i-j}\xi^{\hat{j}}_{\hat{i}} \cdot det_q^{-1} \quad 1 \le i,j \le n,
\end{equation}
where $\hat{k} = \{1, \ldots, k-1, k+1, \ldots, n \}$. $S$
is the antipode for $\AGq$ and makes $\AGq$ a Hopf algebra.

\subsection{Additional Quantum Groups}
In addition to the above mentioned quantum groups, we need some
additional subgroups of $G = \GLq$.

The diagonal subgroup $H_n$ of $\GLq$ is defined by its regular functions
\begin{equation}
    A(H_n) = \mathbb{C}\left[t_1, t_1^{-1}, \ldots,  t_n, t_n^{-1} \right].
\end{equation}
Associated with this commutative Hopf algebra, we have the restriction map $\pi_H: A(G) \rightarrow A(H_n)$ defined by
\begin{equation}
    \pi_H(x_{ij}) = \delta_{i,j}t_i
\end{equation}

The Borel subgroups $B_{+}$ and $B_{-}$ of $\GLq$ consist of the upper and lower triangular matrices
and are defined in terms of their associated Hopf algebras
\begin{align}
    \ABplus &= \mathbb{C} \left[b_{i,j}  \right] ,\;\;\;\; i \le j,\\
    \ABminus &= \mathbb{C} \left[b_{i,j} \right] , \;\;\;\;i \ge j.
\end{align}
These algebras have relations induced from $\AGq$ and we note that the diagonal elements $b_{11},\ldots,b_{nn}$
commute with each other, \cite{NYM93}.  With each of these Hopf algebras we define the restrictions
maps $\pi_{B_+}: \AGq \rightarrow \ABplus$ and $\pi_{B_-}:\AGq \rightarrow \ABminus$ respectively by
\begin{align}
    \pi_{B_+}(x_{ij}) &= \left\{\begin{array}{cc}b_{ij}, & (1 \le i \le j \le n) \\0, & (i > j)\end{array}\right. ,\\
    \pi_{B_-}(x_{ij}) &= \left\{\begin{array}{cc}b_{ij}, & (1 \le j \le i \le n) \\0, & (j > i)\end{array}\right. .
\end{align}

\subsection{Enveloping Algebra $\Uq$}
We recall the quantum universal enveloping algebra $\Uq$  of
$\mathfrak{g} = \mathfrak{gl}(n,\mathbb{C})$ or rather $\mathfrak{sl}(n,\mathbb{C})$\cite{Ka}.  Let $L_n$ be the free $\mathbb{Z}$-module of rank
$n$ with the canonical basis $\{ \epsilon_1, \ldots,\epsilon_n \}$,
i.e. $L_n = \bigoplus\limits_{k=1}^n \mathbb{Z}\epsilon_k$, endowed
with the symmetric bilinear form $\langle \epsilon_i,\epsilon_j
\rangle = \delta_{ij}$. We will define $\alpha_k =
\epsilon_k-\epsilon_{k+1}$.  Additionally, we will identify a
partition $\lambda = (\lambda_1, \ldots, \lambda_n) \in P_n$ with
$\lambda_1\epsilon_1 + \cdots + \lambda_n\epsilon_n \in L_n$. We
will refer to such an element of $L_n$ as a dominant integral
weight.  The fundamental weights are defined by $\Lambda_k =
\epsilon_1 + \cdots + \epsilon_k$ (see \cite{JY}).  Now we define
$\Uq$ as the $\mathbb{C}$-algebra with generators $e_k, f_k$  $(1
\le k < n)$ and $q^{\lambda}$ $\left(\lambda \in
\frac{1}{2}L_n\right)$ with the following relations \cite{NYM93}:
\begin{align}
    q^0 &= 1, \quad q^{\lambda}q^{\mu} = q^{\lambda +\mu}, \\
    q^{\lambda}e_kq^{-\lambda} &= q^{\langle \lambda,\alpha_k \rangle}e_k \quad \left(1 \le k < n \right),\\
    q^{\lambda}f_kq^{-\lambda} &= q^{-\langle \lambda,\alpha_k \rangle}f_k \quad \left(1 \le k < n \right),\\
    e_i f_j - f_j e_i &= \delta_{ij} \frac{q^{\alpha_i}-q^{-\alpha_i}}{q-q^{-1}} \quad \left(1 \le i,j <n \right),\\
    e_i^2 e_j - &\left(q+q^{-1}\right)e_i e_j e_i +e_j e_i^2 = 0 \quad (|i-j| = 1),\\
    f_i^2 f_j - &\left(q+q^{-1}\right)f_i f_j f_i +f_j f_i^2 = 0 \quad (|i-j| = 1),\\
    e_ie_j &=e_je_i, \quad f_if_j=f_jf_i \quad (|i-j| > 1).
\end{align}
We define a coproduct, $\Delta_U$, and a counit, $\varepsilon_U$, on the generators by
\begin{align}
    \Delta_U(q^{\lambda}) &= q^{\lambda} \otimes q^{\lambda}, \quad \varepsilon(q^{\lambda}) = 1, \\
    \Delta_U(e_k) &= e_k \otimes q^{-\alpha_k/2} + q^{\alpha_k/2} \otimes e_k, \quad \varepsilon(e_k) = 0,\\
    \Delta_U(e_k) &= f_k \otimes q^{-\alpha_k/2} + q^{\alpha_k/2} \otimes f_k, \quad \varepsilon(f_k) = 0,
\end{align}
making $\Uq$ a bialgebra.  Additionally, with the antipode $S_{U}$ defined by
\begin{eqnarray}
    S_U(q^{\lambda}) &=& q^{-\lambda},\\
    S_U(e_k) &=& -q^{-1}e_k,\\
    S_U(f_k) &=& -qf_k.
\end{eqnarray}
$\Uq$ becomes a Hopf algebra.

\subsection{$\AGq, \Uq$ Duality}
There exists a well-known dual pairing of  Hopf algebras $\Uq$ and $\AGq$
\begin{equation}
    a(\varphi) \in \mathbb{C},  \hspace{1cm}  a \in \Uq, \varphi \in \AGq
\end{equation}
satisfying the following relations:
\begin{align}
    q^{\lambda}(x_{ij}) &=\; \delta_{i,j}q^{\langle \lambda,\varepsilon_i \rangle}, \quad \lambda \in
    \frac{1}{2}L_n, \quad 1 \le i,j \le n\\
    e_k(x_{ij}) &=\; \delta_{i,k}\delta_{j,k+1}, \quad 1 \le i,j \le n\\
    f_k(x_{ij}) &=\; \delta_{i,k+1}\delta_{j,k}, \quad 1 \le i,j \le n\\
    q^{\lambda}(det_q^m) & = \; q^{m \langle \lambda,\varepsilon_1, \ldots,\varepsilon_n \rangle} \quad m \in \mathbb{Z}\\
    e_k(det_q^m) & = \; f_k(det_q^m) = 0  \quad m \in \mathbb{Z}
\end{align}
We extend these to the rest of $\Uq$ and $\AGq$ by
\begin{align}
    a(\varphi \psi) &=\; \Delta_U(a)(\varphi \otimes \psi)\\
    a(1) &=\; \varepsilon_U(a)\\
    (ab)(\varphi) &= \; (a \otimes b)\Delta(\varphi)\\
    1(\varphi) &=\; \varepsilon(\varphi) \qquad (a,b \in \Uq, \quad \varphi, \psi \in \AGq)
\end{align}
Additionally, we have
\begin{equation}
    S_U(a).\psi = a.S(\psi) \quad a \in \Uq, \psi \in \AGq
\end{equation}
These relations realize a duality between the two Hopf algebras and allows us to regard the elements of
$\Uq$ as linear functionals on $\AGq$ (see \cite{NYM93}).  This duality allows any right $\AGq$-comodule
$V$ (resp. left $\AGq$-comodule $W$) with structure map $R_G:V \rightarrow V \otimes A_q(G)$  (resp.
$L_G:W \rightarrow \AGq \otimes W$) to become a left (resp. right) $\Uq$-module with the following defined action
\begin{align} \label{UqLeftMod}
     a.v = (id \otimes a)R_G(v), \quad a\in \Uq, v \in V, \\
\label{UqRightMod}
     w.a = (a \otimes id)L_G(v), \quad a\in \Uq, w \in W .
\end{align}
More specifically, we already know $\AXq$ is a completely reducible
two-sided $\AGq$-comodule using the comultiplication, $\Delta$, as
the comodule structure map.  As such, it becomes a completely
reducible left and right $\Uq$-module \cite{NYM93, JY}.  We can
describe the left module action of the generators of $\Uq$ on the
generators of $\AXq$ by
\begin{align}
    q^{\lambda}.x_{ij} & = \; x_{ij}q^{\langle \lambda , \varepsilon_j \rangle},\\
    e_k.x_{ij} & = \; x_{i,j-1}\delta_{j,k+1},\\
    f_k.x_{ij} & = \; x_{i, j+1} \delta_{j,k}.
\end{align}
and the right module action as
\begin{align}
    x_{ij}.q^{\lambda} & = \; x_{ij}q^{\langle \lambda , \varepsilon_i \rangle},\\
    x_{ij}.e_k & = \; x_{i+1,j}\delta_{k,i},\\
    x_{ij}.f_k & = \; x_{i-1, j} \delta_{k+1,i}.
\end{align}

\subsection{Relative Invariants}
For an element $\lambda = \sum\limits_{k=1}^n \lambda_k \epsilon_k \in L_n$, let $z^{\lambda}
= \prod\limits_{k=1}^n z_{kk}^{\lambda_k} \in A(B_{\pm})$ and $t^{\lambda}
= \prod\limits_{k=1}^n t_{k}^{\lambda_k} \in A(H)$,
we define the spaces of relative invariants with respect to the subgroups $B_{\pm}$ by (see \cite{NYM93, JY})
\begin{align}
    A(G/B_{+};z^{\lambda}) = \left\{ \varphi \in \AGq;(id \otimes \pi_{B_{+}}) \Delta(\varphi) = \varphi \otimes z^{\lambda} \right\}, \\
    A(B_{-}\backslash G;z^{\lambda}) = \left\{ \varphi \in \AGq;(\pi_{B_{-}} \otimes id) \Delta(\varphi) = z^{\lambda} \otimes \varphi  \right\},
\end{align}
where the restrictions maps $\pi_{\pm}: \AGq
\rightarrow A_q(B_{\pm})$ are defined by $\pi_{B_+}(x_{ij}) =
z_{i,j}$  $ (1\le i \le j \le n)$, $\pi_{B_+}(x_{ij}) = 0$ $(i >
j)$, and $\pi_{B_-}(x_{ij}) = z_{i,j}$  $ (1\le j \le i \le n)$,
$\pi_{B_-}(x_{ij}) = 0$ $(i < j)$.

$A(G/B_{+};z^{\lambda})$ (resp. $A(B_{-}\backslash G;z^{\lambda})$) is a left (resp. right)
$\AGq$-subcomodule of $\AGq$ with structure mapping $\Delta$. It is proved in \cite{NYM93} that,
for a dominant integral weight $\lambda \in P_n$, the space $A(G/B_{+};z^{\lambda}) $ (resp.
$A(B_{-}\backslash G;z^{\lambda})$) gives a realization of the irreducible left (resp. right)
$\AGq$-subcomodule $V_q^L(\lambda)$  (resp. $V_q^R(\lambda)$) of $\AXq$, with highest weight $\lambda$.

\section{Spaces of  $q$-Symplectic Invariants}\label{SympecticInvariants}

\subsection{$\Uspq$}
Here we describe a subalgebra of $\Uq$ that is a quantum deformation of  $\Usp$.
Relative to the standard $n$ dimensional representation of $\Uq$, we identify the generators
$e_k$ of $\Uq$ with $E_{k,k+1}$ and  $f_k$ with $E_{k+1,k}$.  If we let
$\lambda = \lambda_1\epsilon_1 + \cdots + \lambda_{2n}\epsilon_{2n} \in \frac{1}{2}L_{2n}$, then $q^{\lambda}$ is represented by
\begin{equation}
    q^{\lambda_1}E_{11} + q^{\lambda_2}E_{22} + \cdots + q^{\lambda_n}E_{2n,2n}
\end{equation}
We may then inductively generate the other elements, $E_{i,j}$, where $|i-j|>1$,  by
\begin{equation}
    E_{i,j} = E_{i,k}E_{k,j}-E_{k,j}E_{i,k}
\end{equation}
where $i<k<j$ or $j < k < i$ and $E_{i,j}$ and $E_{j,i}$ are
independent of our choice of $k$, see \cite{JY}.

We define the subalgebra $\Uspq$ of $\Uq$ as the subalgebra generated by the following elements:
\begin{align}
    sp_e(i,j) & = \; E_{2i-1,2j} + q^{2(i-j)}E_{2j-1,2i} &1 \le i \ne j \le n \label{gen_e}\\
    sp_e(i,i) & = \; E_{2i-1,2i}  &1 \le i \le n\\
    sp_f(i,j) & = \; E_{2i,2j-1} + q^{2(i-j)}E_{2j,2i-1} & 1 \le i \ne j \le n \label{gen_f}\\
    sp_f(i,i) & = \; E_{2i,2i-1}  &1 \le i \le n\\
    sp_h(i,j) & = \; E_{2i-1,2j-1}-q^{2(i-j)}E_{2j,2i} & 1 \le i,j \le n\label{gen_h}
\end{align}
with $i,j \le n$.  It can be directly shown that the elements of the form
\begin{align}\label{spgen1}
    sp_e(j,j),  \: sp_f(j,j), \quad \text{where } 1 \le j \le n,\\
\label{spgen2}
    sp_e(i,i+1), \: sp_f(i,i+1) ,\quad 1\le i \le n-1
\end{align}
generate $\Uspq$.
\subsection{$q$-Symplectic Invariants}
For a given left (resp. right) $\Uq$-module $V$ (resp. $W$) we
define the $q$-symplectic invariants by 
\begin{align}
    V^K &= \; \left\{ v \in V; sp_e(i,j).v = 0, sp_f(i,j).v = 0  \quad 1\le i,j \le n  \right\} \\
    {}^KW &= \; \left\{ w \in W; w.sp_e(i,j) = 0, w.sp_f(i,j) = 0  \quad 1\le i,j \le n  \right\}
\end{align}
Using the fact that $\AXq$ is a two-sided $\Uq$-module (see \ref{UqLeftMod}, \ref{UqRightMod})
 we define the left and right quantum symplectic invariants in $\AXq$ as
\begin{align}
    \AXq^K &= \; \left\{ \varphi \in \AXq; sp_e(i,j).\varphi = 0, sp_f(i,j).\varphi = 0  \quad 1\le i,j \le n  \right\} \\
    {}^K\AXq &= \; \left\{ \varphi \in \AXq; \varphi.sp_e(i,j) = 0, \varphi.sp_f(i,j) = 0  \quad 1\le i,j \le n  \right\}
\end{align}

The spaces $\AXq^K$ and ${}^K\AXq$ are subalgebras of $\AXq$.  Additionally, we see that $\AXq^K$
is a left $\AGq$-subcomodule of $\AXq$ (similarly ${}^K\AXq$ is a right $\AGq$-subcomodule of $\AXq$).
Equivalently, $\AXq^K$ is a right $\Uq$-submodule of $\AXq$ and ${}^K\AXq$ is a left $\Uq$-submodule of $\AXq$.

\begin{definition} \label{Def2} For $n$ even, the following quadratic elements of $\AXq$ may be defined
\begin{align}\label{Zl}
    z_{i,j}^L &= \sum_{k=1}^{n} q^{(i+j+1-4k)/2}\left( x_{i,2k-1}x_{j,2k}-qx_{i,2k}x_{j,2k-1} \right)\\
    &= \sum_{k=1}^{n/2} q^{(i+j+1-4k)/2} \xi_{2k-1,2k}^{i,j},\nonumber\\
\label{Zr}
    z_{i,j}^R &= \sum_{k=1}^{n} q^{-(i+j+1-4k)/2}\left( x_{2k-1,i}x_{2k,j}-qx_{2k,i}x_{2k-1,j} \right)\\
    &= \sum_{k=1}^{n/2} q^{-(i+j+1-4k)/2} \xi_{i,j}^{2k-1,2k}. \nonumber
\end{align}
\end{definition}

Using the fact
\begin{align}
    e_k.\xi^{i,j}_{r,s}& = \delta_{k,r-1} \xi^{i,j}_{r-1,s} + \delta_{k,s-1}\xi^{i,j}_{r,s-1}\\
    f_k.\xi^{i,j}_{r,s}& = \delta_{k,r} \xi^{i,j}_{r+1,s} + \delta_{k,s}\xi^{i,j}_{r,s+1}
\end{align}
it can be shown that $z_{i,j}^L$ (resp. $z_{i,j}^R$) are annihilated by $sp_e(k,k)$, $sp_e(k,k+1)$,
$sp_f(k,k)$ and $sp_f(k,k+1)$, which is sufficient to show they are annihilated by all $sp_e(k,l)$
and $sp_f(k,l)$ and  therefore $z_{i,j}^L \in \AXq^K$ (resp. $z_{i,j}^R \in {}^K\AXq$)

We denote the subalgebra of $\AXq^K$ (resp. ${}^K\AXq$)  by $\ALq$ (resp. $\ARq$)
generated by $z_{i,j}^L$ (resp. $z_{i,j}^R$).  $\ALq$ is a left $\AGq$-subcomodule of $\AXq^K$
and $\ARq$ is a right $\AGq$-subcomodule of $\AXq^K$.

\begin{theorem}\label{ASRelations}   
The algebras $\ALq$ and $\ARq$ are isomorphic to the algebra $A_q(\mathcal{A})$ generated by
$z_{i,j}$ ($1 \le i, 1 \le j$) with the following relations:
\end{theorem}
\begin{align}z_{i,j} &= -q^{-1}z_{j,i}, \label{AS1}\\
z_{i,l}z_{j,k} &= z_{j,k}z_{i,l}, \label{AS2}\\
z_{i,j}z_{i,k} &= qz_{i,k}z_{i,j}, \label{AS3}\\
z_{i,k}z_{j,l}-&z_{j,l}z_{i,k}  =  \left(q-q^{-1}\right)z_{i,l}z_{j,k}, \label{AS6}\\
z_{i,j}z_{k,l} - &z_{k,l}z_{i,j} = \left(q-q^{-1}\right)z_{i,k}z_{j,l}
-q\left(q-q^{-1}\right)z_{i,l}z_{j,k} \label{AS7},
\end{align}
where $ i < j < k < l $.

\bigskip
Using Eq. (\ref{AS6}) we may rewrite Eq. (\ref{AS7}) as
\begin{equation}\label{AS8}z_{i,j}z_{k,l}-z_{k,l}z_{i,j} = qz_{j,l}z_{i,k} - q^{-1}z_{i,k}z_{j,l}\end{equation}
\hfill$\Box$

The definitions of these generators also imply
\begin{equation}\label{AS9}z_{i,i} =0\end{equation}

\subsection{Quantum Antisymmetric Matrices}
If we denote by $\mathcal{A}$, the vector space of $n \times n$ antisymmetric matrices with basis
\begin{equation}
    B_{\mathcal{A}} = \left\{E_{i,j}-E_{j,i}\left| 1<i<j \le n  \right.  \right\}
\end{equation}
then  $dim(\mathcal{A})= n(n-1)/2$.  We observe that $Hom_{Alg}(A_q(\mathcal{A}),\mathbb{C})$
is the set of $n \times n$ matrices with restrictions imposed by the relations Eq. (\ref{AS1}) and Eq. (\ref{AS9}).
If we denote $Hom_{Alg}(A_q(\mathcal{A}),\mathbb{C})$  by $\Aq$, and treat it as a vector space
(in other words we are ignoring multiplication) we see its basis is
\begin{equation}
    B_{\Aq}= \left\{E_{i,j}-qE_{j,i}\left| 1<i<j \le n  \right.  \right\}
\end{equation}
where $dim(\Aq)= n(n-1)/2$ and we have $\Aq \simeq \mathcal{A}$ as vector spaces.
We may think of $\Aq$ as the quantum analog of the antisymmetric matrices.

 \subsection{Quantum Pfaffian} \label{Pfaffianq}
If $A = (a_{i,j}) \in Mat(2n,\mathbb{C})$ is an antisymmetric matrix, it can be written as
\begin{equation}
    A= \begin{bmatrix}
            0 & a_{1,2} & \cdots & a_{1,2n}\\
            -a_{1,2} & 0 & \cdots & a_{2,2n}\\
            \vdots & \vdots & \ddots& \vdots\\
            -a_{1,2n} & -a_{2,2n} & \cdots & 0
        \end{bmatrix}
\end{equation}
and there exists a polynomial $f$ in $\mathbb{Z}[x_{ij}]$ such that $f^2(A) =
det(A)$, \cite{Ja}. This
polynomial is called the Pfaffian, denoted $\pfc$, and we
write
\begin{equation}
    \pfc^2 (A) = det (A)
\end{equation}
Moreover, if $B = (b_{i,j}) \in Mat(2n,\mathbb{C})$ and we define $A$ by
\begin{equation}\label{ClassicalA}
    a_{i,j} = det\begin{bmatrix}b_{i,1} & b_{i,2}\\ b_{j,1} & b_{j,2} \end{bmatrix}
        +det\begin{bmatrix}b_{i,3} & b_{i,4}\\ b_{j,3} & b_{j,4} \end{bmatrix}
        + \cdots + det\begin{bmatrix}b_{i,2n-1} & b_{i,2n}\\ b_{j,2n-1} & b_{j,2n} \end{bmatrix}
\end{equation}
then $A$ is antisymmetric and we have $\pfc (A) = det(B)$, \cite{Ja}.

To construct an explicit formula for $\pfc$ we can define an index set $\Pi$, consisting of all
ordered, $2$-partitions of $2n$.  In other words,\begin{equation}
    \Pi = \left\{ (i_1,j_1)(i_2,j_2) \ldots (i_{n},j_{n}) \; ;\; i_k < j_k  \text{ and } i_k < i_{k+1} \right\}
\end{equation}
For example, if $n=4$ we have
\begin{equation}
    \Pi = \left\{ (1,2)(3,4) , (1,3)(2,4) , (1,4)(2,3) \right\}
\end{equation}
We can associate the elements of $\Pi$ with elements of the symmetric group $\mathfrak{S}_{2n}$ in the following manner
\begin{equation}
    \pi \sim \begin{bmatrix} 1& 2 & 3 & 4 & \cdots & 2n \\
                i_1 & j_1 & i_2 & j_2 & \cdots & j_{n} \end{bmatrix} \in  \mathfrak{S}_n
\end{equation}
for $\pi = \left\{ (i_1,j_1)(i_2,j_2) \ldots (i_{n},j_{n}) \right\}$.
This allows us to define $sgn(\pi)$  and $l(\pi)$.  If $A = (a_{i,j})$
is an antisymmetric matrix we can then write
\begin{equation}
    \pfc(A) =\sum_{\pi \in \Pi} sgn(\pi)a_{\pi}
    = \sum_{\pi \in \Pi} sgn(\pi) a_{i_1j_1} a_{i_2j_2} \cdots a_{i_{n},j_{n}}
\end{equation}

\begin{example}
As an example, when $2n=4$
\begin{equation}
    \pfc(A) = a_{1,2}a_{3,4} - a_{1,3}a_{2,4} + a_{1,4}a_{2,3}
\end{equation}
\end{example}

Before we construct a quantum analog of the Pfaffian, we note that the quantum antisymmetric
generators $z_{i,j}^L$ (resp. $z_{i,j}^R$), defined by Eq. (\ref{Zl})  (resp. Eq. (\ref{Zr})),
are in fact quantum analogs of Eq. (\ref{ClassicalA}).  Additionally, we have already noted
that $Z = (z^L_{i,j})$ is a quantum antisymmetric matrix with the relation $z^L_{i,j} =
-\frac{1}{q}z^L_{j,i}$ for $i<j$.  We now use the same index set $\Pi$, to define the quantum Pfaffian as
\begin{equation}
    \pfq(Z) =\sum_{\pi \in  \Pi} (-q)^{l(\pi)} z^L_{\pi}\\
    = \sum_{\pi \in  \Pi} (-q)^{l(\pi)} z^L_{i_1j_1} z^L_{i_2j_2} \cdots z^L_{i_{n},j_{n}}.
\end{equation}

\begin{remark} An inductive definition of quantum Phaffian was given in \cite{S}. One can show that
our definition matches with Strickland's.
\end{remark}

\begin{example}
As an example, when $2n=4$
\begin{align}
    \pfq(Z) =& \; z^L_{1,2}z^L_{3,4} - qz^L_{1,3}z^L_{2,4} + q^2 z^L_{1,4}z^L_{2,3}\end{align}
\end{example}

\begin{theorem}\label{pfaffianprop}For every positive even $2n$, $\pfq(Z) = det_q(X)$.
\end{theorem}


\begin{proof}To show this equality, we will prove that $\pfq$ is simultaneously a highest and lowest weight vector for the right action of $\Uq$.  This will show $\pfq$ to be a scalar multiple of $(det_q)^c$ for some $c \in \mathbb{Z}_+$.

To begin, we let $k$ be a positive integer such that $1 \le k < n$.  Since the right action of generators of $\Uq$ on products of elements of $\AXq$ can be described by  \cite{JY},
\begin{align}
	\phi \psi.e_k & =\; (\phi \otimes \psi).(e_k \otimes q^{-a_k/2} + q^{a_k/2} \otimes e_k )\\
	\phi \psi.f_k & =\; (\phi \otimes \psi).(f_k \otimes q^{-a_k/2} + q^{a_k/2} \otimes f_k )
\end{align}
we may expand this notation to describe the following right action of $e_k$ on the components of $\pfq$ as
\begin{align}
z^L_{a_1b_1}z^L_{a_2b_2}\cdots z^L_{a_{n/2}b_{n/2}}.e_k =& \; z^L_{a_1b_1}.e_k \otimes z^L_{a_2b_2}.q^{-\alpha_k/2} \otimes \cdots  \otimes z^L_{a_{n/2}b_{n/2}}.q^{-\alpha_k/2} \nonumber \\
& + z^L_{a_1b_1}.q^{\alpha_k/2} \otimes z^L_{a_2b_2}.e_k \otimes \cdots  \otimes z^L_{a_{n/2}b_{n/2}}.q^{-\alpha_k/2} \nonumber \\
& \quad \vdots \nonumber \\
& + z^L_{a_1b_1}.q^{\alpha_k/2} \otimes z^L_{a_2b_2}.q^{\alpha_k/2} \otimes \cdots  \otimes z^L_{a_{n/2}b_{n/2}}.e_k
\end{align}
and
\begin{align}
z^L_{a_1b_1}z^L_{a_2b_2}\cdots z^L_{a_{n/2}b_{n/2}}.f_k = & \; z^L_{a_1b_1}.f_k \otimes z^L_{a_2b_2}.q^{-\alpha_k/2} \otimes \cdots  \otimes z^L_{a_{n/2}b_{n/2}}.q^{-\alpha_k/2} \nonumber \\
& + z^L_{a_1b_1}.q^{\alpha_k/2} \otimes z^L_{a_2b_2}.f_k \otimes \cdots  \otimes z^L_{a_{n/2}b_{n/2}}.q^{-\alpha_k/2} \nonumber \\
& \quad \vdots \nonumber \\
& + z^L_{a_1b_1}.q^{\alpha_k/2} \otimes z^L_{a_2b_2}.q^{\alpha_k/2} \otimes \cdots  \otimes z^L_{a_{n/2}b_{n/2}}.f_k
\end{align}

Additionally, each of these $\ALq$ generators is a sum of quantum $2$-minor determinants (see Eq. (\ref{Zl})) in which the indices $i$ and $j$ of $z^L_{ij}$ define the rows for each of these quantum $2$-minor determinants .  As such, the right action of $e_k$ and $f_k$ on these generators can be described by  the following,
\begin{align}\label{Z_e_rightaction}
	z^L_{i,j}.e_k &= q^{-1/2}\left( \delta_{i,k}z^L_{k+1,j} + \delta_{j,k} z^L_{j,k+1}\right),\\
\label{Z_f_rightaction}
	z^L_{i,j}.f_k &= q^{1/2}\left( \delta_{i,k+1}z^L_{k,j} + \delta_{j,k+1} z^L_{j,k}\right)
\end{align}
and the right action of $q^{\alpha/2}$ and $q^{-\alpha/2}$ are described by
\begin{align}\label{Z_qplus_rightaction}
	z^L_{i,j}.q^{\alpha_k/2} &= q^{1/2(\delta_{i,k}-\delta_{i,k+1}+\delta_{j,k}-\delta_{j,k+1})}z^L_{i,j},\\
\label{Z_qminus_rightaction}
	z^L_{i,j}.q^{-\alpha_k/2} &= q^{1/2(-\delta_{i,k}+\delta_{i,k+1}-\delta_{j,k}+\delta_{j,k+1})}z^L_{i,j}.
\end{align}
For example
\begin{equation}
	z^L_{3,4}.q^{\alpha_4/2} = q^{1/2}z^L_{3,4}
\end{equation}

Before we give a detailed description of the action of $e_k$ on $\pfq$, we show how the components of $\Pi$ may be paired, relative to the value of $k$.  Since the components of $\pfq$ are indexed by all ordered $2$-partitions, this will allow us to group the components of $\pfq$ in a way that the right action of $e_k$ (and $f_k$) will annihilate the pairs.

We first fix $k \in \mathbb{Z}$ such that $1 \le k < n$.  Now if we choose any of the ordered $2$-partitions, say $\pi = (a_1,b_1)(a_2,b_2)\cdots(a_{n},b_{n})$,  it must have an index $r$, containing $k$ and an index $s$ containing $k+1$.  In other words, there exist $r$ and $s$ such that
\begin{equation}
	k \in (a_r,b_r) \quad\mbox{and}\quad
	k+1 \in (a_s,b_s)
\end{equation}
This fixes $r$ and $s$.  Also contained in the $(a_r,b_r)$ and $(a_s,b_s)$ pairs are two other integers, $u$ and $v$ such that $u < v$. If it happens that $r = s$, in other words, there exists $(a_r,b_r)$ such that $(a_r,b_r) = (k,k+1)$ then we will not pair it with another $2$-partition.  We will show later how the right action of $e_k$ and $f_k$ already annihilate it.

\begin{example}\label{PartitionExampleA}
Suppose $2n = 8$ and we fix $k = 5$.  One of the ordered $2$-partitions of $\Pi$ is $(1,3)(2,6)(4,8)(5,7)$.  In this case we see that $r = 4$ and $s = 2$.  We then designate $u=2$ and $v=7$.
\end{example}

Now, with $r$ and $s$ still fixed, and for the designated $u$ and $v$, there are precisely three possibilities describing how $k$, $k+1$, $u$ and $v$ can be ordered.  These are:
\begin{align}
	k < k+1 < u < v\\
	u < k < k+1 < v\\
	u < v < k < k+1
\end{align}

For each of these possibilities we have the following,
\begin{itemize}
	\item $k < k+1 < u < v$
	
	In this case, if $r \ne s$, there is another $2$-partition, $\hat{\pi}$  identical to $\pi$ except in the $r^{th}$ and $s^{th}$ pairs, $u$ and $v$ are switched.
	\begin{align}
	\pi = & \; (a_1,b_1) \cdots (k,u) (k+1,v) \cdots (a_{n},b_{n})\\
	\hat{\pi} = & \;(a_1,b_1) \cdots (k,v) (k+1,u) \cdots (a_{n},b_{n})
	\end{align}
  	If $r = s$ then we have
	\begin{align}
	\pi = & \; (a_1,b_1) \cdots (k,k+1)\cdots (u,v) \cdots (a_{n/2},b_{n/2})
	\end{align}
	
  \item $u < k < k+1 < v$

  In this case, if $r \ne s$, there is a second partion $\hat{\pi}$ identical to $\pi$ except in the $r^{th}$ and $s^{th}$ pairs, $k$ and $k+1$ are switched.
  \begin{align}
  	\pi = & \; (a_1,b_1) \cdots (u,k) \cdots (k+1,v) \cdots (a_{n},b_{n})\\
	\hat{\pi} = & \;(a_1,b_1) \cdots (u,k+1) \cdots (k,v) \cdots (a_{n},b_{n})
  \end{align}
  	If $r = s$ then we have
	\begin{align}
	\pi = & \; (a_1,b_1) \cdots (u,v)\cdots (k,k+1) \cdots (a_{n},b_{n})
	\end{align}

  \item $u < v < k < k+1$

  In this case, if $r \ne s$,  there is a second partion $\hat{\pi}$ identical to $\pi$ except in the $r^{th}$ and $s^{th}$ pairs, $k$ and $k+1$ are switched.
  \begin{align}
	\pi = & \; (a_1,b_1) \cdots (u,k) \cdots (v,k+1) \cdots (a_{n},b_{n})\\
	\hat{\pi} = & \;(a_1,b_1) \cdots (u,k+1) \cdots (v,k) \cdots (a_{n},b_{n})
  \end{align}
  	If $r = s$ then we have
	\begin{align}
	\pi = & \; (a_1,b_1) \cdots (u,v)\cdots (k,k+1) \cdots (a_{n},b_{n})
	\end{align}

\end{itemize}
\begin{example}
Continuing with the previous example (Example \ref{PartitionExampleA}), with $2n = 8$, $k = 5$ and $2$-partition $(1,3)(2,6)(4,8)(5,7)$, the other $2$-partition with which this would be paired is $(1,3)(2,5)(4,8)(6,7)$.
\end{example}
Using this construction, we see that after fixing $k$, we may exhaustively list all of the ordered $2$-partitions of $\Pi$, identifying each $2$-partition as containing a pair $(k,k+1)$ or as being one of the pairs just described.

This allows us to  write $\pfq$ as a sum of components of the form
\begin{equation}
	(-q)^* z^L_{a_1,b_1} \cdots z^L_{k,k+1} \cdots z^L_{a_{n},b_{n}}
\end{equation}
or which appear in pairs such as
\begin{align}
	(-q)^* z^L_{a_1,b_1} \cdots z^L_{k,u}  z^L_{k+1,v} \cdots z^L_{a_{n},b_{n}} \nonumber \\
	(-q)^{*+1} z^L_{a_1,b_1} \cdots z^L_{k,v} z^L_{k+1,u} \cdots z^L_{a_{n},b_{n}}
\end{align}
or
\begin{align}
	(-q)^* z^L_{a_1,b_1} \cdots z^L_{u,k} \cdots z^L_{k+1,v} \cdots z^L_{a_{n},b_{n}} \nonumber \\
	(-q)^{*+1} z^L_{a_1,b_1} \cdots z^L_{u,k+1} \cdots z^L_{k,v} \cdots z^L_{a_{n},b_{n}}
\end{align}
or
\begin{align}
	(-q)^* z^L_{a_1,b_1} \cdots z^L_{u,k} \cdots z^L_{v,k+1} \cdots z^L_{a_{n},b_{n}} \nonumber \\
	(-q)^{*+1} z^L_{a_1,b_1} \cdots z^L_{u,k+1} \cdots z^L_{v,k} \cdots z^L_{a_{n},b_{n}},
\end{align}
where $(-q)^{*}$ represents an appropriate power of $(-q)$ determined by $(a_1b_1)(a_2b_2)\cdots(a_{n}b_{n})$.  The right action of $e_k$ can now be calculated. In the first case,  we have the index that contains $(k,k+1)$ and we have

\begin{align}
	q^*  z^L_{a_1,b_1} &\cdots z^L_{k,k+1} \cdots z^L_{a_{n},b_{n}} . e_k \nonumber \\
	=& \; (z^L_{a_1,b_1}.e_k) \cdots (z^L_{k,k+1}.q^{-\alpha_k/2}) \cdots (z^L_{a_{n},b_{n}}.q^{-\alpha_k/2}) \nonumber \\
	& +\; (z^L_{a_1,b_1}.q^{\alpha_k/2}) \cdots (z^L_{k,k+1}.e_k) \cdots (z^L_{a_{n},b_{n}}.q^{-\alpha_k/2}) \nonumber \\
	& +\; (z^L_{a_1,b_1}.q^{\alpha_k/2}) \cdots (z^L_{k,k+1}.q^{\alpha_k/2}) \cdots (z^L_{a_{n},b_{n}}.e_k) \nonumber \\
	=& \; (0) \cdots (z^L_{k,k+1}.q^{-\alpha_k/2}) \cdots (z^L_{a_{n},b_{n}}.q^{-\alpha_k/2}) \nonumber \\
	& +\; (z^L_{a_1,b_1}.q^{\alpha_k/2}) \cdots (0) \cdots (z^L_{a_{n},b_{n}}.q^{-\alpha_k/2}) \nonumber \\
	& +\; (z^L_{a_1,b_1}.q^{\alpha_k/2}) \cdots (z^L_{k,k+1}.q^{\alpha_k/2}) \cdots (0) \nonumber \\
	= & \; 0
\end{align}

In the next case, with the indexes of the paired $2$-partitions containing $(k,u) (k+1,v)$ and $(k,v) (k+1,u)$, the right action of $e_k$ can be seen to be zero as well, by using Eq. (\ref{Z_e_rightaction}), Eq. (\ref{Z_qplus_rightaction}), and Eq. (\ref{Z_qminus_rightaction}). In fact all remaining cases are treated
similarly, and we get that
\begin{equation}
	\pfq.e_k = 0
\end{equation}
A similar argument shows
\begin{equation}
	\pfq.f_k = 0
\end{equation}
Since $\pfq$ is an element of $\AXq$ annihilated by the right action of all $e_k$ and $f_k$, $1 \le k < n$, $\pfq$ must be generated by $det_q$.  By comparing degree and coefficients, we see $\pfq(Z) = det_q(X)$.

\end{proof}

We extend the notation slightly and define

\begin{equation}\label{PfaffianI}
    \pfq(Z) ^{I} = \sum_{\pi \in \Pi^I} (-q)^{l(\pi)} z^L_{\pi}
\end{equation}
where $I = \{1,2, \ldots, r\}$,  $r<n$, $r$ is even and $\Pi^{I}$ is the set of ordered $2$-partitions of $I$. The proof above also shows that  $\pfq(Z) ^{I}$ is annihilated on the right by all $f_k$ ($k<n$) and by all $e_k$ except for $r<k<n$.  As such $\pfq(Z) ^{I}$ is still a highest weight vector under the right action of  $\Uq$ and, because it is an element constructed from left symplectic invariant generators, it  provides a realization of an element in $\AXq^K \cap A(B_{-}\backslash G ; z^{\Lambda_r})$.
\subsection{Decomposition of ${}^K\AXq)$ and $\AXq^K$}

We show the decomposition of ${}^K\AXq)$ as a right $\AGq$-comodule (resp. left $\Uq$-module) and the
decomposition of $\AXq)^K$ as a left $\AGq$-comodule (resp. right $\Uq$-module.  To perform this
decomposition, several preliminary propositions are presented, along with the introduction of
some notational conventions.  First some notation:

We define the map $\phi$ from the power set of $\{1,2,3,\ldots,n\}$ into the power set of $\{1,2,3,\ldots,n\}$ by
\begin{equation}
    \phi(A) = \bigcup_{\alpha \in A}\{ 2\alpha-1, 2\alpha  \}
\end{equation}
for example $\phi(\{ 1,3,4,5\}) = \{1,2,5,6,7,8,9,10\}$.
We will use  $\phi$  to construct indices for the rows and columns of quantum minor determinants used
in $q$-symplectic invariants and then to describe a specific set of dominant weights as
\begin{equation}\label{PAn}
    P_n^{\mathcal{A}} = \left\{\lambda \in P_n \; ; \; \lambda = (\mu_1, \mu_1, \mu_2, \mu_2, \ldots , \mu_{n}, \mu_{n}), \; \mu \in P_{n} \right\}
\end{equation}
For example     $(4,4,4,4,3,3,2,2,2,2,1,1) \in P^{\mathcal{A}}_{12}$.

One of the key ideas used in the decomposition of ${}^K\AXq$ and
$\AXq^K$ is presented in the following proposition (cf. \cite{JY}).

\begin{proposition}\label{reciprocityprop}
Let $\mu \in P_n$ be a dominant integral weight and  $V^R_q(\mu)$ be the irreducible left $\Uq$
submodule with highest weight $\mu$.  Then the space of the $q$-symplectic invariants in $V^R_q$
has the dimension equal to the multiplicity of $V$ in ${}^K\AXq)$.
\end{proposition}

\begin{proof} To decompose the algebra ${}^K\AXq)$ as a right $\AGq$-comodule
(or left $\Uq$ module), it suffices to find the singular weight vectors, i.e. the weight
vectors $\phi \in {}^K\AXq$ such that $e_k.\phi = 0$ for $k= 1, \ldots, n-1$.   Since such
a singular vector $\varphi$ is contained in the space ${}^K\AXq \cap A(X/B^{+}; z^{\lambda})$
for some dominant integral weigh $\lambda \in L_n$, and generates an irreducible right
$\AGq$-comodule with highest weight $\lambda$. Thus if there are $m_{\lambda}$ singular
weight vectors of weight $\lambda$ in ${}^K\AXq$, then the irreducible right $\AGq$-comodule
isomorphic to $V^R_q(\lambda)$ occurs $m_{\lambda}$ times in the decomposition of ${}^K\AXq$.
On the other hand, a singular vector $\varphi$ in ${}^K\AXq \cap A(X/B^{+}; z^{\lambda})$
is regarded as a left $q$-symplectic invariant in $V^L_q$ (i.e. annihilated on left).
Since  $V^L_q(\lambda)$ and $V^R_q(\lambda)$ are dual to each other, the dimension of the
space of $q$-symplectic invariants coincides.
\end{proof} \bigskip


Next, we show by construction, the existence of a left invariant in the left $\Uq$-module $A(B_{-}\backslash G ;z^{\lambda})$.
We build this left invariant from elements of the following form
\begin{equation}\label{aR}
    a_r^R = \sum_J q^{-2|J|}\xi^{1,\ldots,2r}_{\phi(J)}
\end{equation}
where the sum is over all $J$ such that $\#J = r$ and $J \subseteq \{1, 2, \ldots, n \}$.
$|J|$ represents the sum of the elements of $J$.
As such $a_r^R \in \AXq^K \cap A(B_{-}\backslash G ; z^{\lambda_r})$.

\begin{lemma}(Existence)\label{ExistenceLemma}
For $\lambda = \sum \limits_{r=1}^{n} m_{2r}\Lambda_{2r}$, i.e, $\lambda \in P^{\mathcal{A}}_n$,
$A(B_{-}\backslash G;z^{\lambda})$ contains a left $q$-symplectic invariant.
\end{lemma}

\begin{proof}
Suppose $\lambda = \sum \limits_{r=1}^{n} m_{2r} \Lambda_{2r} $, then we define
\begin{equation}\label{aRlambda}
    a^R_{\lambda} =\prod_{r=1}^{n}\left(a^R_r \right)^{m_{2r}}
\end{equation}
where $a^R_r$ is defined by Eq. (\ref{aR}).  We see by its construction, $a^R_{\lambda} \in
\AXq^K \cap A(B_{-}\backslash G ; z^{\lambda})$.  As such, each right $\AGq$-comodule
$V^R_q(\lambda)$ has a $q$-symplectic invariant.
\end{proof}

\begin{lemma}(Nonexistence)\label{NonExistenceLemma}
There does not exist a left $q$-symplectic invariant in the irreducible right $\Uq$-submodule
$V_q^L(\lambda)$ if $\lambda \notin P^{\mathcal{A}}_n$.
\end{lemma}

\begin{proof}
$\AXq^K$ is a right $\Uq$-submodule of $\AXq$.  As such, it has its own decomposition into
irreducible right $\Uq$-submodules indexed by
$\lambda \in P_n$ , where $\lambda$ is a dominant integral weight
\begin{equation}
    \AXq^K = \bigoplus\limits_{\lambda} V_q^L(\lambda)
\end{equation}
Each $V_q^L(\lambda)$ is a highest weight module  \cite{NYM93}.
Each of these highest weight modules has a
realization of $A(G/B^{+}; z^{\lambda})$ with highest weight vector of the form
\begin{equation}\label{highestweight}
v_{\lambda} = \left(\xi_{1,\ldots,s}^{1,\ldots,s}\right)^{m_s}
\left(\xi_{1,\ldots,s-1}^{1,\ldots,s-1} \right)^{m_{s-1}}
\cdots \left(\xi_{1}^{1} \right)^{m_{1}}
\end{equation}
However, because, the elements of $\AXq^K$ are annihilated on the left by all $e_k$ and $f_k$
where $k$ is odd, then for the highest weight vector $v_{\lambda}$,
it must be true that $\lambda = \sum\limits_{r=1}^n m_r\Lambda_r$ where $m_r =  0$ when $r$ is odd.  Thus,
\begin{equation}
    \AXq^K = \bigoplus\limits_{\lambda \in P^{\mathcal{A}}_n} V_q^L(\lambda)
\end{equation}
\end{proof} \bigskip

\begin{lemma}(Uniqueness)\label{UniquenessLemma}
The multiplicity of  $V_q^R(\lambda)$ an irreducible right $\Uq$-module, in the decomposition
of $\AXq^K$ is exactly one.
\end{lemma}

\begin{proof}As mentioned earlier, by Proposition \ref{reciprocityprop},
the multiplicity of $V_q^R(\lambda)$ in the decomposition of $\AXq^K$ is
equal to the number of left $q$-symplectic invariants in $A(B_{-}\backslash G;z^{\lambda})$.
Let $v^K$ be a non zero left invariant in
$A(B_{-}\backslash G;z^{\lambda})$, as such,  it can be written as a linear combination of
weight vectors from the standard basis of
$A(B_{-}\backslash G;z^{\lambda})$  \cite{NYM93}.

However, since $A(B_{-}\backslash G;z^{\lambda})$ is a highest weight vector space,
there must be at least one basis (weight) vector,
$\eta$, in the composition of $v^K$ for which there are no higher weight vectors in $v^K$.
In other words
\begin{equation}
    v^K = \eta \oplus v_1 \oplus \cdots v_j
\end{equation}
where the weights of $v_1, \ldots , v_j$ are less than or equal to that of $\eta$.
As such, $\eta$ must be annihilated by all $e_k$, where $k < n$ and $k$ is odd.
Additionally, the elements of the form
\begin{align}
sp_f(i,i+1) =& \; e_{2i} + q^{-2}\left( f_{2i-1}f_{2i}f_{2i+1} - f_{2i}f_{2i-1}f_{2i+1}\right. \nonumber \\
    & \left. - f_{2i+1}f_{2i-1}f_{2i} + f_{2i+1}f_{2i}f_{2i-1} \right)
\end{align}
where $1 \le i < n$, also annihilate $v^K$, and this in turn requires that $\eta$
also be annihilated by all $e_k$, where  $k < n$ and $k$ is even.
Therefore $\eta$ must be a highest weight vector of $A(B_{-}\backslash G;z^{\lambda})$,
but this vector is unique up to constant multiple,
because $A(B_{-}\backslash G;z^{\lambda})$ is a highest weight module.  So
\begin{equation}
    \eta =  cv_{\lambda}, \quad c \in \mathbb{C}
\end{equation}
where $v_{\lambda}$ is defined by Eq. (\ref{highestweight}).  This tells us that
any non-zero left $q$-symplectic invariant in
$A(B_{-}\backslash G;z^{\lambda})$ must be written as
\begin{equation}
    cv_{\lambda} \oplus w_1 \oplus \cdots w_j, \quad c  \in \mathbb{C}, c \ne 0
\end{equation}
 where $w_1, \ldots, w_j$ are lower weight vectors of $A(B_{-}\backslash G;z^{\lambda})$.

Now assume there is more than one left quantum $q$-symplectic invariant in
$A(B_{-}\backslash G;z^{\lambda})$, say $v^K$ and $w^K$.
Each of these may be written as a sum of standard basis elements, each including a
non-zero term for the highest weight vector $v_{\lambda}$.
In other words, they may be written as
\begin{equation}
    v^K = c_0v_{\lambda} +c_1v_1 + c_2v_2 + \cdots +c_iv_i,\quad c_0 \ne 0
\end{equation}
\begin{equation}
    w^K = k_0v_{\lambda} +d_1v_1 + d_2v_2 + \cdots +d_jv_j, \quad k_0 \ne 0
\end{equation}
Since the linear combination of any left $q$-symplectic invariant is also a left
$q$-symplectic invariant then it must be
true that $k_0v^K - c_0w^K$ is also a left $q$-symplectic invariant in
$A(B_{-}\backslash G;z^{\lambda})$.
If $k_0v^K - c_0w^K \ne 0$ then we have a contradiction to the requirement that
any left $q$-symplectic invariant
in $A(B_{-}\backslash G;z^{\lambda})$ has a nonzero $v_{\lambda}$ component.
On the other hand, if $k_0v^K - c_0w^K = 0$
then $w^K$ is a constant multiple of $v^K$.  Therefore, any left $q$-symplectic
invariant in $A(B_{-}\backslash G;z^{\lambda})$
is unique up to a constant multiple. \end{proof} \bigskip

The following proposition summarizes Lemmas \ref{ExistenceLemma}, \ref{NonExistenceLemma},
and \ref{UniquenessLemma}.
\begin{proposition}\label{ExistenceAndUniqueness}
The space of $q$-symplectic invariants in the right $\AGq$-comodule
$V^R_q(\mu)$ is one dimensional
if and only if $\mu = \sum \limits_{r=1}^{n} m_{2r}\Lambda_{2r}$, in other words, $\mu \in P^{\mathcal{A}}_n$.
Otherwise there are no $q$-symplectic invariants in $V^R_q$
\end{proposition}

By Proposition \ref{reciprocityprop} we may then summarize our results with the following theorem
\begin{theorem}
The irreducible decomposition of $\AXq^K$ as a right $\Uq$-module is given by
    \begin{equation}
    \AXq^K = \bigoplus_{\lambda \in P^{\mathcal{A}}_{n}}V^L_q(\lambda)
    \end{equation}
similarly ${}^K\AXq$, as a left $\Uq$-module has the irreducible decomposition
    \begin{equation}
    {}^K\AXq = \bigoplus_{\lambda \in P^{\mathcal{A}}_{n}}V^R_q(\lambda)
    \end{equation}
Where $P^{\mathcal{A}}_{n}$ is defined by Eq. (\ref{PAn}).
\end{theorem}

\begin{proposition}\label{ALandARDecomp}
The space $\ALq = \AXq^K$, (resp. ${}^K\AXq = \ARq$).  As such, $\ALq$ (resp. $\ARq$)
also have the decompositions as a right (resp. left) $\Uq$-modules,
    \begin{align}
    \ALq = \bigoplus_{\lambda \in P^{\mathcal{A}}_{n}}V^L_q(\lambda) \\
    \ARq = \bigoplus_{\lambda \in P^{\mathcal{A}}_{n}}V^R_q(\lambda)
    \end{align}
\end{proposition}

\begin{proof}From its definition, we already have $\ALq \subseteq \AXq^K$.
The elements, ${Pf_q}^I$ described by
Eq. (\ref{PfaffianI}) provide a formula for explicitly constructing a left $\Uspq$ invariant in
$A(B_{-}\backslash G; z^{\lambda})$ for any $\lambda$.    As such, $\AXq^K \subseteq \ALq$,
and we have $\ALq = \AXq^K$.
\end{proof}

\section{Bi-invariants}
In this section we define a subalgebra of $\AXq$ by the intersection of $\ARq$ and $\ALq$.  Defined in this way,
this space is annihilated on the left and right by $\Uspq$.  We then proceed to show that this algebra is really
$\mathbb{C}[s_1, \ldots, s_{n}]^{\mathfrak{S}_{n}}$, the symmetric algebra of $n$ variables.  To start, we define $A_{ZP}$,  as
\begin{equation}\label{A_ZP}
A_{ZP} = \ARq \cap \ALq = \bigoplus_{m=0}^{\infty} A_{ZP,2m},
\end{equation}
Recall, the polynomials of $\ARq$ and $\ALq$ have even degree so it has the natural grading into the subspaces $A_{ZP,2m}$.

Now we define
\begin{equation}
    E_r = \sum_{I,J} q^{2(|I|-|J|)}\xi^{\phi(I)}_{\phi(J)}, \quad 1 \le r \le n
\end{equation}
where the summation runs over all subsets $I$ and $J$ of $\{1, \ldots, \frac{n}{2} \}$ and $\#I = \#J = r$.
Here, $|I|$ and $|J|$ are the sums of the elements of $I$ and $J$ respectively.

\begin{lemma}$E_r \in A_{ZP,2r}$
\end{lemma}

\begin{proof}If we examine the component of $E_r$ that is obtained by holding
$I$ fixed at $I = \{1,2,\ldots,r\}$,
we see that this component is precisely $a^R_r$, defined in Eq. (\ref{aR}).  As such,
this component is
invariant under the left action of $\Uspq$.  The remaining components of $E_r$
(the components obtained by fixing $I$ at other values)
can be obtained by the right action of $\Uq$ on $a^R_r$.  Since $\AXq^K$ is a right
submodule of $\AXq$ these other
components of $E_r$ must also be left invariant.  Thus, $E_r \in \AXq^K$.  Similarly,
we see that the component of $E_r$
 associated with the fixed $J = \{1,2,\ldots,n\}$ is in ${}^K\AXq$ and likewise the
 other components of $E_r$ can be
 obtained by the left action of $\Uq$.  Thus, $E_r \in {}^K\AXq$.  Since $E_r$ has
 degree $2r$ (by its construction)
 and $E_r \in \ARq \cap \ALq$, it follows that $E_r \in A_{ZP,2r}$.
\end{proof}

\begin{theorem}
The algebra $A_{ZP}$ is generated by  $E_{r} (1 \le r \le n)$ and the algebra $A_{ZP}$ is isomorphic to the
algebra of symmetric polynomials of $n$ variables;
\begin{equation}
    \pi : A_{ZP} \; \tilde{\rightarrow} \; \mathbb{C}[s_1, \ldots , s_{n}]^{\mathfrak{S}_{n}}
\end{equation}
\end{theorem}

\begin{proof}Because of the decomposition given in
Proposition \ref{ALandARDecomp}, the dimension of the bi-invariant
space associated with each $\lambda \in P^{\mathcal{A}}_n$ must be
exactly one.  Since the degree of the polynomial in each of these
bi-invariant spaces is $\sum \limits_{k=1}^n \lambda_k$, the
dimension of $A_{ZP,2m}$ can then be calculated as the number of
partitions in $P^{\mathcal{A}}_n$ of  $2m$. As these partitions are
in $P^{\mathcal{A}}_n$ we may also consider this as the number of
partitions of $m$ whose number of parts is less than or equal to
$n$.  Adopting the notation of Jing and Yamada \cite{JY} we denote
this by $p_{n}(m)$.

Consider the restriction of the projection map $\pi$  to $A_{ZP}$
\begin{equation}
    \pi'_H:  A_{ZP} \rightarrow A_{+}(H),
\end{equation}
where $A_{+}(H) = \mathbb{C}[t_1, \ldots , t_{n} ]$.  Then $Ker(\pi'_H) = \bigoplus \limits_{r=0}^\infty Ker(\pi'_{H,2r})$,
where
\begin{equation}
    \pi'_{H,2r}: A_{ZP,2r} \rightarrow A_{2r}(H).
\end{equation}
Similar to the proof by \cite{JY}, the monomials
$E_{r_1}E_{r_2}\ldots E_{r_k}$  $(r_1 \le r_2 \le \ldots \le r_k)$
have the degree $2(r_1 + r_2 + \ldots + r_k)$ and are linearly
independent over $\mathbb{C}$. As such the space of degree $2m$
spanned by these monomials has dimension $p_{n}(m)$.  This shows
that  the space $A_{ZP}$ is generated by $E_r$ $(1 \le r \le n)$.

Additionally, the map $\pi'_{H,2r}$ acts on the generators of $A_{ZP}$ in the following manner
\begin{align*}
    \pi'(E_r) & = \pi' \left( \sum_{I} \xi^{\phi(I)}_{\phi(I)} \right)\\
    & = \sum_{I} \left(t_{2i_1-1}t_{2i_1}\right)\left(t_{2i_2-1}t_{2i_2}\right)\cdots\left(t_{2i_r-1}t_{2i_r}\right)\ne 0
\end{align*}
where the sum runs over all subsets $I$  of $\{1, 2, \ldots, \frac{n}{2} \}$ and $\#I = r$,
thus $Ker(\pi'_{H,2r}) = (0)$.
Another way of viewing this is that each of these $E_r$ has monomials which are products of diagonal elements.
As such, $\pi(E_r) \ne 0$ for $1 \le r \le n$.  Thus we have the isomorphism
\begin{align*}
    A_{ZP} &\cong \mathbb{C}[(t_1t_2),(t_3t_4), \ldots , (t_{n-1}t_n)]^{\mathfrak{S}_{n}}\\
    &\cong \mathbb{C}[s_1,s_2, \ldots , s_{n}]^{\mathfrak{S}_{n}}
\end{align*}
where we let $s_i = t_{2i-1}t_{2i}$.
\end{proof}

\section{Spherical functions and Symmetric polynomials}

Through the isomorphism in Theorem (4.1)our $q$-zonal polynomials
are basis elements in the ring of symmetric polynomials, and they
are clearly $q$-deformation of the zonal polynomials defined on
$GL(2n,\mathbb C)/Sp(2n,\mathbb C)$.  We describe the relation with
Macdonald polynomials \cite{Mac}.

Macdonald polynomials are special orthogonal basis of the
commutative algebra \newline$\mathbb
Q(q,t)[x_1,\ldots,x_n]^{\mathfrak S_n}$, where $q$ and $t$ are two
parameters.  To describe them we consider the following shift
operator $T_{u,x_i}$ by
$$
(T_{u,x_i}f)(x_1,\ldots,x_n)=f(x_1,\ldots,ux_i,\ldots,x_n)
$$
for each $f\in\mathbb Q(q,t)[x_1,\ldots,x_n]$.  Let $X$ be another
indeterminate and define
$$\aligned D(X,q,t)&=\Delta^{-1}\sum\limits_{w\in\mathfrak
S_n}\epsilon(w)z^{w\delta}\prod\limits^n_{i=1}(X+t^{(w\delta)_i}T_{q,x_i})\\
&=\sum\limits^n_{r=0}D_rX^{n-r},
\endaligned$$
where $\delta=(n-1,n-2,\ldots,1,0)$ and
$$\Delta=\prod\limits_{1\leq i<j\leq n}(x_i-x_j)$$
is the Vandermonde determinant in $x_1,\ldots,x_n$.  It follows
immediately that $D_0=1$ and
$$D_1=\sum\limits^n_{i=1}(\prod\limits_{j\neq i}{tx_i-x_j\over
x_i-x_j})T_{q,x_i}\, .$$

Macdonald showed that for each partition
$\lambda=(\lambda_1,\ldots,\lambda_n)$ there is a unique symmetric
polynomial $P_\lambda(x;q,t)$ satisfying the two conditions (4.5 -
4.6):
\begin{equation}P_\lambda=m_\lambda+\sum\limits_{\mu<\lambda}u_{\lambda\mu}m_\mu
\end{equation}
where $u_{\lambda\mu}\in\mathbb Q(q,t)$ and $m_\mu=x_1^{\mu_1}\ldots
x_n^{\mu_n}+\ldots$ is the monomial symmetric polynomial;
\begin{equation}D_1P_\lambda=(\sum\limits^n_{i=1}q^{\lambda_i}t^{n-i})P_\lambda.
\end{equation}

Moreover Macdonald proves that $P_\lambda$ is also an eigenfunction
for all the difference operators $D_r$, and
\begin{equation} D(X;q,t)P_\lambda=\prod\limits^n_{i=1}(X+t^{n-i}q^{\lambda_i})P_\lambda.
\end{equation}

The polynomial $P_\lambda(x;q,t)$ is called the Macdonald polynomial
associated with the partition $\lambda$.  In particular,
$P_\lambda(x;q,q)$ is the famous Schur polynomial;
$\lim\limits_{t\rightarrow 1}P_\lambda(x;t^2,t)$ is the zonal
polynomial.

\begin{proposition} Under the isomorphism
$\pi:A_{zp}\longrightarrow\mathbb C[z_1,\ldots,z_{n}]^{\mathfrak
S_{n}}$, the q-zonal polynomial in $V_q(\lambda)$ is a constant
multiple of the Macdonald polynomial $P_\lambda(z;q^2,q^{-4})$.
\end{proposition}

The general case of quantum spherical functions was studied by Noumi \cite{N}
using quantum groups and Letzter \cite{L} using quantum enveloping algebras.  In
the following we will outline a different approach to understand the relationship between
symmetric functions and quantum invariants. First of all let's study the $q$-difference
operators on $V_q(2\lambda)$.

Recall the center of the quantized universal enveloping algebra
$U_q({\mathfrak sl}_{n-1})$ is generated by the following $n-1$
elements \cite{FRT}.
$$
c_k=\sum\limits_{\sigma,\sigma'\in\mathfrak
S_n}(-q)^{l(\sigma)+l(\sigma')}l^{(+)}_{\sigma_1,\sigma_1'}\cdots
l^{(+)}_{\sigma_k\sigma'_k}l^{(-)}_{\sigma_{k+1}\sigma'_{k+1}}\cdots
l^{(-)}_{\sigma_n\sigma'_n},\quad k=1,\cdots,n-1$$ where
$L^{(\pm)}=(l_{ij}^{(\pm)})$ is the upper (lower) triangular
defining matrix for the quantum algebra $U_q(sl_{n-1})$ in the FRT
formulation \cite{FRT} and $l(\sigma)=\#\{i<j|\sigma_i>\sigma_j\}$.  We
only remark that the elements $l_{ij}^{(\pm)}$ are analogs of
Weyl-generators for $U_q(\mathfrak{sl}_{n-1})$.  In particular
$$l^{(\pm)}_{ii}=q^{\pm\epsilon_i},$$

The algebra $U_q({\mathfrak sl}_{n-1})$ acts on $GL_q(n,\mathbb C)$
as q-difference operators, thus the center of
$U_q(\mathfrak{sl}_{n-1})$ acts on modules $V_q(2\lambda)$ as scalar
operators. In particular our $q$-zonal polynomials are simultaneous
eigenfunctions of these $q$-difference operators.

\begin{theorem} For $1\leq k\leq n-1$, the central
element $c_k$ acts on the irreducible $U_q(sl_n)$-module
$V(\lambda)$ as a scalar multiplication by
$$q^{2|\lambda|+\binom n2+k(n-1)}[k]![n-k]!
(\sum_{1\leq i_1<\cdots<i_k\leq
n}q^{-2\lambda_{i_1}-\cdots-2\lambda_{i_k}
+2(i_1-n)+\cdots+2(i_k-n)}),
$$
where $|\lambda|=\lambda_1+\cdots+\lambda_n$.
\end{theorem}

\begin{proof}Pick a lowest weight vector $v_0$ in $V(\lambda)$ with
the weight
$-\lambda=-{\lambda_1}\epsilon_1-\cdots-{\lambda_n}\epsilon_n$. Note
that the generators $l^{(+)}_{ij},l_{ji}^{(-)}(i<j)$ belong to the
so-called strict upper and lower Borel subalgebra generated by $e_i$
and $f_i$ ($i=1, \cdots n-1$) respectively. The element
$l^{(-)}_{\sigma_{k+1}\sigma'_{k+1}}\cdots
l^{(-)}_{\sigma_n\sigma'_n}$ kills $v_0$ unless $\sigma_{k+1}=
\sigma'_{k+1}$, \ldots, $\sigma_{n}= \sigma'_{n}$. But
$\sigma_{1}\leq\sigma'_{1}$, \ldots, $\sigma_{k}\leq\sigma'_{k}$, so
one must have $\sigma=\sigma'$ in the action of $c_{n-k}$ on $v_0$.
We thus have
\begin{align*} c_{k}v_0&=\sum\limits_{\sigma\in\mathfrak
S_n}q^{2l(\sigma)}q^{-\lambda_{\sigma_1}-\cdots-\lambda_{\sigma_{k}}
+\lambda_{\sigma_{k+1}}+\cdots+\lambda_{\sigma_n}}v_0\\
&=q^{|{\lambda}|}\sum\limits_{\sigma\in\mathfrak
S_n}q^{2l(\sigma)-2\lambda_{\sigma_{1}}-\cdots-2\lambda_{\sigma_k}}v_0.
\end{align*}

Consider the Young subgroup $\mathfrak S_k\times \mathfrak S_{n-k}$
of $\mathfrak S_n$. We can choose its left coset representatives to
be the elements $\tau$ such that $\tau_1<\cdots<\tau_k$,
$\tau_{k+1}< \cdots < \tau_n$. Recall that an inversion of the
permutation $\tau$ is a pair $(ij)$ such that $i<j$ and
$\tau_i>\tau_j$. By construction the inversions of $\tau$ may only
take place among $(ij)$ where $i\leq k$ and $j\geq k+1$. For each $i
(\leq k)$, there are $\tau_i-1$ natural numbers less than $\tau_i$,
and $i-1$ of them already appear before $\tau_i$ in the permutation.
So there are $\tau_i-i$ inversions of $\tau$ in the form $(ij)$,
which implies that $l(\tau)=\sum_{i=1}^k (\tau_i-i)$.

Let $\tau\sigma$ be the general element in $\mathfrak S_n$ where
$\sigma=\sigma_1\sigma_2\in \mathfrak S_k\times\mathfrak S_{n-k}$.
In the sequence $(\tau\sigma(1), \ldots, \tau\sigma(k),
\tau\sigma(k+1), \ldots, \tau\sigma(n))$ we divide the inversions of
$\tau\sigma$ into three parts: the inversions among the first $k$
numbers, those among the last $n-k$ numbers, and the inversions
between the first $k$ numbers and the last $n-k$ numbers. The second
part $(\tau\sigma(k+1), \ldots, \tau\sigma(n))$
$=(\tau\sigma_2(k+1), \ldots, \tau\sigma_2(n))$ has $l(\sigma_2)$
inversions as $\tau$ preserves the order of $k+1, \ldots, n$,
similarly the first part $(\tau\sigma(1), \ldots, \tau\sigma(k))$
$=(\tau\sigma_1(1), \ldots, \tau\sigma_1(k))$ has $l(\sigma_1)$
inversions among them. Observe that we are free to switch the
numbers in each part when considering the inversions between the
first part and the second part, thus the number of inversions of
this type are exactly $l(\tau)$. Therefore we have
\begin{align*}
l(\tau\sigma_1\sigma_2)&=l(\tau)+l(\sigma_1)+l(\sigma_2)\\
&=l(\sigma_1)+l(\sigma_2)+\sum_{i=1}^k (\tau_i-i),
\end{align*}
where $\sigma_1\in \mathfrak S_k, \sigma_2\in \mathfrak S_{n-k},
\tau\in \mathfrak S_n/(\mathfrak S_k\times \mathfrak S_{n-k})$.

Now let's return back to the action $c_kv_o$. Using the
invariance of $\lambda_{\sigma(1)}+\cdots +\lambda_{\sigma(k)}$ under
$\mathfrak S_k\times\mathfrak S_{n-k}$, we have that
\begin{align*} &c_{k}v_0
=q^{2|\lambda|}\sum\limits_{\tau, \sigma_1,
\sigma_2}q^{2l(\tau\sigma_1\sigma_2)-2\lambda_{\tau\sigma_1\sigma_2
(1)}-\cdots -2\lambda_{\tau\sigma_1\sigma_2
(k)}}v_0\\
&=q^{2|\lambda|}\sum\limits_{\tau, \sigma_1,
\sigma_2}q^{2l(\tau\sigma_1\sigma_2)-2\lambda_{\tau (1)}-\cdots
-2\lambda_{\tau(k)}}v_0\\
&=q^{2|\lambda|} \sum_{\sigma_1\in\mathfrak S_k}q^{2l(\sigma_1)}
\sum_{\sigma_2\in\mathfrak S_{n-k}}q^{2l(\sigma_2)}
\sum\limits_{\tau}
q^{2l(\tau)-2\lambda_{\tau (1)}-\cdots-2\lambda_{\tau (k)}}v_0\\
&=q^{2|\lambda|+\binom k2+\binom{n-k}2}[k]![n-k]!
(\sum_{\tau(1)<\cdots<\tau(k)}q^{-2\lambda_{\tau
(1)}-\cdots-2\lambda_{\tau (k)}
+2(\tau(1)-1)+\cdots+2(\tau(k)-k)})v_0\\
&=q^{2|\lambda|+\binom n2+k(n-1)}[k]![n-k]!\\
&\hskip 1in \cdot(\sum_{1\leq\tau(1)<\cdots<\tau(k)\leq
n}q^{-2\lambda_{\tau (1)}-\cdots-2\lambda_{\tau (k)}
+2(\tau(1)-n)+\cdots+2(\tau(k)-n)})v_0
\end{align*}
where we have used the well-known identity
$\sum\limits_{\sigma\in\mathfrak S_n}q^{2l(\sigma)}=q^{\binom
n2}[n]!$ (cf. \cite{B}). \end{proof}

Now we restrict ourselves to the case of irreducible highest
$U_q(sl_{2n})$-module $V(\tilde{\lambda})$ such that
$\tilde{\lambda}=\tilde{\lambda}_1\epsilon_1+\cdots+\tilde{\lambda}_{2n}\epsilon_{2n}$
and $\lambda_{2i-1}=\lambda_{2i}=\lambda_i$ for $i=1, \ldots, n$. It is
also a lowest weight module with the lowest weight
$-\tilde{\lambda}$.

\begin{theorem}
The bi-invariant function inside $V(\tilde{\lambda})$, restricted to
the ring of symmetric functions,  is the Macdonald symmetric
function $P_{\lambda}(q^2, q^4)$.
\end{theorem}
\begin{proof} It follows from the theorem in the case of $U_q(sl_{2n})$-module
$V(\tilde{\lambda})$ that
\begin{align*}
c_{1}v_0&=q^{2|\tilde{\lambda}|+\binom
{2n}2+2(2n-1)}[2]![2n-2]!\sum_{1\leq i\leq 2n}q^{-2\tilde{\lambda}_i
+2(i-2n)}v_0\\
&=q^{4|\lambda|+\binom {2n}2+2(2n-1)-1}[2]^2[2n-2]!\sum_{1\leq i\leq
n}q^{-2\lambda_i +4(i-n)}v_0.
\end{align*}
In other words, the quantum Casimir operator $c_1$ agrees with
Macdonald operator $D_1(q^2, q^4)$ or $D_1(q^{-2}, q^{-4})$ on the space. We note that the leading term of the spherical functions, when restricted to
the zonal part, are exactly the leading term of the Macdonald spherical function
$P_{\lambda}(q^2, q^4)$ (which also agrees with Schur funtion $s_{\lambda}$). Hence
the eigenfunction restricted to the ring of symmetric functions are
Macdonald symmetric function $P_{\lambda}(q^2, q^4)$. Similarly the action of the
higher difference operators are given by
\begin{align*}
c_{2k}v_0&=q^{2|\tilde{\lambda}|+\binom {2n}2+2k(2n-1)}[2k]![2n-2k]!\\
&\hskip 0.5in \cdot(\sum_{1\leq\tau(1)<\cdots<\tau(2k)\leq
n}q^{-2\tilde{\lambda}_{\tau (1)}-\cdots-2\tilde{\lambda}_{\tau
(2k)}
+2(\tau(1)-2n)+\cdots+2(\tau(2k)-2n)})v_0\\
\end{align*}
\end{proof}

The last identity plus the same idea also gives that
\begin{corollary} The
restriction of $c_{k}$ to the ring $\mathbb
C[z_1,\ldots,z_n]^{\mathfrak S_n}$ is exactly the difference
operator $D_k(q^2, q^4),\;k=1,\ldots, n$ up to a constant.
\end{corollary}

\centerline{\bf Acknowledgments}
NJ gratefully acknowledges the partial support of Max-Planck Institut f\"ur Mathematik in Bonn, Simons Foundation grant 198129, and NSFC grant 10728102 during this work.

\begingroup
\baselineskip12pt
\parskip12pt



N.J.: Department of Mathematics, North Carolina State University,
Raleigh, NC 27695, USA and School of Sciences, South China University of
Technology, Guangzhou, Guangdong 510640, China

jing@math.ncsu.edu

R.R.: Department of Mathematics, Gonzaga University, Spokane, WA
99258, USA

rayr@gonzaga.edu

\end{document}